\documentclass[11pt, a4]{article}
\usepackage{amssymb, amsfonts,amsmath, amsthm,amsbsy}
\usepackage{mdwlist} %for compact list
\usepackage{setspace,graphicx} 
\usepackage[all]{xy}  
\usepackage{hyperref}
\usepackage{enumerate}

\newtheorem{theorem}{Theorem}[section]
\newtheorem{lem}[theorem]{Lemma}
\newtheorem{prop}[theorem]{Proposition}
\newtheorem{cor}[theorem]{Corollary}

\theoremstyle{definition}
\newtheorem{defn}[theorem]{Definition}

\newtheorem{re}[theorem]{Remark}

\newcommand{\dom}[0]{\ensuremath{\textup{Dom}}}
\newcommand{\id}[0]{\ensuremath{\textup{Id}}}
\newcommand{\image}[0]{\ensuremath{\textup{Image}}}
\newcommand{\Ker}[0]{\ensuremath{\textup{Ker}}}

\newcommand{\sgn}[0]{\ensuremath{\textup{sgn}}}

\newcommand{\paral}{/\kern-0.3em/}
\def\parals_#1{/\kern-0.3em/_{\!#1}}
\def\paralss_#1^#2{/\kern-0.3em/_{\!#1}^{\!#2}}

\begin{document}
%%%%%%%%%%%%%%%
\begin{center}
\Large{\textbf{It\^o-Wiener chaos and the Hodge decomposition on an abstract Wiener space}}
\par
\normalsize{By {Yang Yuxin}}
\par
\today
\end{center}
\textbf{Abstract:}
Using the structure of the Boson-Fermion Fock space and an argument taken from \cite{bieliavsky2005symplectic}, 
we give a new proof of the triviality of the $L^2$ cohomology groups on an abstract Wiener space, alternative to that given by Shigekawa \cite{shigekawa1986rham}.   
We apply some representation theory of the symmetric group to characterise the spaces of exact and co-exact forms in their Boson-Fermion Fock space representation.  
\\
\textbf{MSC2010:} %\textbf{AMS Classification Codes} : 
58A14, 58A12, 58B05, 20C35, 60H30.
\\
\textbf{Acknowledgements:} 
The author is immensely indebted to Professor David Elworthy and Dr Mikhail Korotyaev for valuable discussions.  Thanks are also due to Professor John Rawnsley for pointing out to his earlier result, and to Dr Mark Wildon for pointing out the reference by W. Hamernik.  This research is partially funded by an Early Career Fellowship from the Warwick Institute of Advanced Study.
%%%%%%%%%%%%%%
\section{Introduction}
\label{sec:intro}
%%%%%%%%%%%%%%
With a Weitzenb\"ock formula for the Hodge Laplacian with positive curvature, I. Shigekawa \cite{shigekawa1986rham} proved the vanishing of the $L^2$ de Rham cohomology on an abstract Wiener space. 
This was a first step in developing a Hodge theory on infinite dimensional manifolds, a goal first set by L. Gross \cite{gross1967potential} in his pioneering work on infinite dimensional potential theory. 
Shigekawa's definitive treatment of the linear case provides guidance for the study of nonlinear cases. For example, Fang and Franchi \cite{fang1997differentiable} used the It\^o map to transfer Shigekawa's results from the Wiener space to the path spaces over a compact Lie group with a bi-invariant metric. 
However, the problem of developing a Hodge theory in more general infinite-dimensional manifolds remains open.  
\par
We present here a different interpretation of Shigekawa's results based on the well-known isometric isomorphism between the $L^2$ Gaussian space and the symmetric Fock space. This isomorphism is known as the It\^o-Wiener chaos expansion \cite{ito1951multiple} on the classical Wiener space, which expresses any square integrable functional as a sum of multiple stochastic integrals. 
We extend this isomorphism to skew-symmetric $L^2$ vector fields, which in turn allows us to transform the study of the Hodge theory on an abstract Wiener space $(E, H, \mu)$ to the study of vector subspaces of  $H^{\otimes n}$ of the form $H^{\odot k}\otimes H^{\wedge (n-k)}$ for varying integers $k$. 
Such Boson-Fermion Fock spaces have been studied by A. Arai \cite{arai1993dirac}, among others.
Here $H^{\otimes n}$ denotes the completed $n$-th tensor powers of the Hilbert space $H$, with $H^{\odot n}$ and $H^{\wedge n}$ its subspaces of symmetric and skew-symmetric tensors, respectively, both completed using the Hilbert space cross norm inherited from $H^{\otimes n}$. 
\par
In this context, the de Rham complex of exterior differential forms on our abstract Wiener space can be seen to restrict to two long exact sequences of vector spaces of the form $H^{\odot k}\otimes H^{\wedge (n-k)}$.  
The exactness of such sequences has been studied in \cite{bieliavsky2005symplectic} in a purely algebraic setting.  
The result of \cite{bieliavsky2005symplectic}
involves an identity analogous to Shigekawa's Weitzenb\"ock formula, and implies the Hodge decomposition and the triviality of the de Rham cohomology groups on the abstract Wiener space. 
Using the representation theory of symmetric groups, we prove further a direct-sum decomposition of the tensor product $H^{\odot k}\otimes H^{\wedge (n-k)}$ considered as a subspace of $H^{\otimes n}$, which shows geometrically how the exact sequences split. 
Such a decomposition of supersymmetric Fock spaces is of independent interest; algebraically it goes back to W. Hamernik \cite{hamernik1976specht} for the finite dimensional case.  
\par
The organisation of this article is as follows. After a quick review of the basic notation and Shigakawa's results in Section \ref{sec:notation}, we explain the Boson-Fermion Fock-space interpretation of these results in Section \ref{sec:fock}. 
In Section \ref{sec:decomp}, we present a representation-theoretic proof of the direct-sum decomposition of the tensor product $H^{\odot k}\otimes H^{\wedge (n-k)}$. 
%%%%%%%%%%%%%
\section{Notation and Shigekawa's Results}
\label{sec:notation}
%%%%%%%%%%%%%
On an abstract Wiener space $(E, H, \mu)$, the natural notion of differentiability is $H$-differentiability, and correspondingly we consider $H$-differential-forms, i.e., sections of dual bundle of exterior powers of $H$. Since we are primarily interested in the $L^2$ theory, we concentrate on $L^2$ differential $q$-forms, denoted $L^2\Gamma(H^{\wedge q})^*$. 
\par
The exterior derivative $d_q$ and its adjoint $d_q^*$ are defined, as in Shigekawa \cite{shigekawa1986rham}, by 
\[ 
d_q=(q+1)A_{q+1}D, 
\quad \mbox{ and }\quad
d_{q}^*=D^*, 
\]
respectively, with $D$ denoting the $H$-derivative, and $D^*$ its adjoint; all the operators here are closed and densely defined. 
Note that the standard skew-symmetrisation operator $A_q$ on $q$-tensors, given by 
\[
A_q(h_1\otimes\dots\otimes h_q)=\frac{1}{q!}\sum_{\rho\in \mathfrak{S}_q}\sgn(\rho) (h_{\rho(1)}\otimes\dots\otimes h_{\rho(q)}), \quad h_1, \cdots, h_q \in H, 
\]
is extended to give the alternating map (denoted by the same symbol) on linear functionals defined on tensor products, so given any $\phi\in L(H^{\otimes q}; \mathbb{R})$, 
\[
A_q\phi(h_1, \dots, h_q)=\frac{1}{q!}\sum_{\rho\in \mathfrak{S}_q}\sgn(\rho) \phi(h_{\rho(1)}, \dots,h_{\rho(q)}). 
\]
The summation here is over all $q!$ elements of the symmetric group $\mathfrak{S}_q$, which consists of all permutations of $\{1, \dots,q\}$.  
\par
Shigekawa \cite{shigekawa1986rham} derived the following Weitzenb\"ock identity: %(Proposition 3.1) 
\begin{equation}
\label{eq:weitzenbock}
\Delta_q = L + q \id,
\end{equation}
where $\Delta_q=d_q^*d_q+ d_{q-1}d_{q-1}^*$ is the Hodge-Kodaira Laplacian on $q$-forms, and $L=D^*D$ is the Ornstein-Uhlenbeck operator. 
Denote by $\mathfrak{h}_q$ the set of all the harmonic forms of degree $q$, i.e., $\phi\in L^2\Gamma (\wedge^q H)^*$ is in $\mathfrak{h}_q$ if $\phi\in\dom(\Delta_q)$ and $\Delta_q\phi = 0 $. 
\begin{theorem}[Shigekawa \cite{shigekawa1986rham}]
\label{th:shigekawa}
$L^2\Gamma (\wedge^q H)^* = \image(d_{q-1})\oplus \image(d_q^*)\oplus \mathfrak{h}_q$, where 
\begin{description}
\item[1.] $\image(d_{q-1}) = \Ker(d_q)$;  
\item[2.] $\image(d_q^*) = \Ker(d_{q-1}^*)$; and 
\item[3.] 
$\mathfrak{h}_q=\{ 0 \}$ for $q\ge 1$, and $\mathfrak{h}_0=\{\mbox{constant functions}\}$.
\end{description} 
Equivalently, the following sequences are exact: 
\[
0 \rightarrow \mathbb{R}
\xrightarrow{\iota} L^2\Gamma \mathbb{R}
\xrightarrow{d_0} L^2\Gamma H^*
\xrightarrow{d_1}\cdots 
\xrightarrow{d_{q-1}} L^2\Gamma (H^{\wedge q})^*
\xrightarrow{d_{q}} \cdots 
\]
and \[
0 \leftarrow \mathbb{R}
\xleftarrow{\iota^*} L^2\Gamma \mathbb{R}
\xleftarrow{d_0^*} L^2\Gamma H^*
\xleftarrow{d_1^*}\cdots 
\xleftarrow{d_{q-1}^*} L^2\Gamma (H^{\wedge q})^*
\xleftarrow{d_{q}^*} \cdots 
\]
where the maps $\iota: \mathbb{R}\rightarrow L^2\Gamma \mathbb{R}$ and $\iota^*: L^2\Gamma \mathbb{R}\rightarrow\mathbb{R}$ are defined by 
\[
\iota (c)(x) = c, \quad c\in \mathbb{R}, x\in E,
\]
and 
\[
\iota^*f=\mathbb{E}f, \quad f\in 
L^2(E; \mathbb{R}).
\]
respectively.
\end{theorem} 
%%%%%%%%%%%%%%
\section{A Fock-Space Interpretation}
\label{sec:fock}
%%%%%%%%%%%%%%
We can work with skew-symmetric vector-fields instead of differential forms, which is justified by the Riesz correspondence between the Hilbert space $H$ and its dual $H^*$, and similarly between the completed tensor powers $H^{\otimes n}$ and its dual $(H^{\otimes n})^*$. 
The consequent correspondence between $L^2$ $H$-forms and skew-symmetric $L^2$ $H$-vector fields 
\[
L^2\Gamma (H^{\wedge q})^* \cong
L^2\Gamma H^{\wedge q} \cong 
L^2(E, \mu; \mathbb{R})\otimes H^{\wedge q}  
\]
allows us to define, corresponding to the exterior derivative $d_q$ on $q$-forms, an operator 
$d_{q}^\sharp:\dom(d_{q}^\sharp)\subset L^2\Gamma H^{\wedge q}\rightarrow L^2\Gamma H^{\wedge (q+1)}$ 
on skew-symmetric $q$-vector fields by  
\[
d_{q}^\sharp u_{\sigma}=(d_q u^\sharp_{\sigma})^\sharp, \quad u\in L^2\Gamma H^{\wedge q}, \,\sigma\in E,
\]
where $u^\sharp \in L^2\Gamma (\wedge^q H)^*$ is given by $u^\sharp_{\sigma}(h)=<u_{\sigma}, h>_{H^{\wedge q}}$, and $u \in \dom(d_{q}^\sharp)$ iff $u^\sharp\in \dom(d_q)$. 
Similarly for $d_q^*$, we define an operator $d_{q}^{*\sharp}$ by
\[
d_{q}^{*\sharp} u_{\sigma}=(d_{q}^* u^\sharp_{\sigma})^\sharp, \quad u\in L^2\Gamma H^{\wedge (q+1)}, \,\sigma\in E,
\]
where $u \in \dom(d_q^{*\sharp})$ iff $u^\sharp\in \dom(d_q^{*})$. 
We note that $d_q^{*\sharp} = d_q^{\sharp*}$, and that all the operators here are closed and densely defined. 
\par
We follow Shigekawa \cite{shigekawa1986rham} to define the exterior product $\wedge$ by 
\[
h_1\wedge\cdots\wedge h_q = q! \,A_{q}(h_1\otimes\cdots\otimes h_q), \quad h_1, \cdots, h_q\in H, 
\]
Similarly the symmetric product $\odot$ is defined by 
\[
h_1\odot\cdots\odot h_k = k! \, S_{k}(h_1\otimes\cdots\otimes h_k),  \quad h_1, \cdots, h_k\in H, 
\]
with the symmetrisation operator $S_k$ on $k$-tensors given by 
\[
S_k(h_1\otimes\dots\otimes h_k)=\frac{1}{k!}\sum_{\rho\in \mathfrak{S}_k} (h_{\rho(1)}\otimes\dots\otimes h_{\rho(k)}), \quad h_i \in H.
\]
As in Shigekawa \cite{shigekawa1986rham}, we use the following inner product for $H^{\wedge q}$ (instead of the usual inner product induced from $H^{\otimes q}$)
\[
<h_1\wedge\cdots\wedge h_q, g_1\wedge\cdots\wedge g_q> = \det(<h_i, g_j>), \quad h_1, \cdots, h_q, g_1, \cdots, g_q \in H, 
\]
and similarly for $H^{\odot k}$,
\[
<h_1\odot\cdots\odot h_k, g_1\odot\cdots\odot g_k> = \sum_{\rho\in \mathfrak{S}_k} \prod_{i=1}^k <h_i, g_{\rho(i)}>.
\]
\par
The well-known isometry between the symmetric (or Boson) Fock space over $H$, 
$\mathbf{F}_s(H) = \bigoplus_{k=0}^{\infty} H^{\odot k}$, and the $L^2$ Gaussian space, $L^2(E, \gamma; \mathbb{R})$, 
\[
\Psi: \mathbf{F}_s(H) \cong L^2(E, \mu; \mathbb{R}), 
\]
is defined by (see, e.g., \cite{neveu1968processus})
\begin{equation}
\label{eq:psi}
\Psi (\exp\odot h)= \exp(I(h)-\frac{1}{2}\|h\|^2), \quad h\in H, 
\end{equation}
where 
\[
\exp\odot h = 
\sum_{k=0}^{\infty}\frac{1}{k!}h^{\odot k}\in\bigoplus_{k=0}^{\infty} H^{\odot k}. 
\]
Recall that these exponential vectors $\{\exp\odot h\}_{h\in H}$ form a total subset of $\mathbf{F}_s(H)$.  
The isometry $\Psi$ can be extended to skew-symmetric $L^2$ $H$-vector fields to obtain isomorphisms 
\[
\Psi_q: \mathbf{F}_s(H) \otimes  H^{\wedge q}\cong  L^2(E, \mu; \mathbb{R})\otimes H^{\wedge q}, 
\]
if we set, for each $q\in \mathbb{N}$,  
\begin{equation}
\label{eq:psi_q}
\Psi_q = \Psi\otimes \id_{H^{\wedge q}}. 
\end{equation}  
\par
To relate to the Hodge theory on our abstract Wiener space, we first observe that, if the following diagram commutes  
\begin{equation}
\label{figure:psi_q}
\xymatrix{
{\bigoplus_{k=0}^{\infty} H^{\odot k}\otimes  H^{\wedge q}} 
\ar[r]^{\Psi_q } 
\ar[d]_{\breve d_{q}}
& {L^2(E, \mu; \mathbb{R})\otimes  H^{\wedge q}} 
\ar[d]_{ d_{q}^\sharp}
\ar[r]^{\cong}
&  {L^2\Gamma(H^{\wedge q}} )^*
\ar[d]_{ d_{q}}
\\
{
\bigoplus_{k=1}^{\infty} H^{\odot k-1}\otimes  H^{\wedge (q+1)}} 
\ar[r]^{\Psi_{q+1} } 
& {L^2(E,\mu; \mathbb{R})\otimes H^{\wedge (q+1)}}
\ar[r]^{\cong}
&  {L^2\Gamma(H^{\wedge (q+1)}} )^*,
} 
\end{equation}
it reduces the study of the exterior derivative $d_q$ and its adjoint $d_q^*$ on the abstract Wiener space $(E, H, \mu)$, which are on the right side of the diagram, to that of their counterparts on the extended Fock spaces, on the left side.  
We start by defining the operator $\breve d_q$ on $\mathbf{F}_s(H) \otimes H^{\wedge q}$, and its adjoint operator $\breve d_q^{*}$ on $\mathbf{F}_s(H) \otimes  H^{\wedge (q+1)}$, via their restrictions on each $k$-th chaos. 
\begin{defn}
\label{def:breve_d_q}
Given $k, q \in\mathbb{N} $, and $h_1, \cdots, h_k, x_1, \cdots, x_q, x_{q+1} \in H$, 
we set 
$\breve d_{q} (x_1\wedge \cdots \wedge x_q )=0$, 
\begin{eqnarray*}
&& \breve d_{q} (h_1\odot \cdots \odot h_k \otimes  x_1\wedge \cdots \wedge x_q )\\
&=& \sum_{j=1}^k h_1\odot \cdots \odot \hat h_j \odot \cdots \odot h_k\otimes  h_j \wedge x_1\wedge\cdots\wedge x_q,
\end{eqnarray*}
where $\hat{}$ denotes omission, and 
\begin{eqnarray*}
& & \breve d_{q}^* (h_1\odot \cdots \odot h_k \otimes  x_1\wedge \cdots \wedge x_{q+1})\\
&=& \sum_{i=1}^{q+1} (-1)^{i-1} h_1\odot \cdots \odot h_k\odot x_i \otimes x_1\wedge\cdots\wedge\hat x_i\wedge\cdots\wedge x_{q+1}.
\end{eqnarray*}
\end{defn}
For brevity, we use the shorthand notation $H_{k,q}= H^{\odot k}\otimes H^{\wedge q}$  
for this subspace of $H^{\otimes (k+q)}$. 
It is clear that, up to constant multiples, the operator $\breve d_q$ is the composition of the inclusion of $H_{k,q}$ into $H^{\otimes (k+q)}$ followed by the projection onto $H_{k-1,q+1}$.  
It is not difficult to verify the following properties: 
\begin{enumerate}[(a)]
\item $\breve d_{q+1}\circ \breve d_{q} =0, \quad  \breve d_{q}^*\circ \breve d_{q+1}^* =0$; 
\item $\breve d_q|_{H_{k,q}}= (q+1) \id_{H^{\odot (k-1)}}\otimes A_{q+1}, \quad \breve d_q (H_{k,q}) \subset H_{k-1, q+1}$; and 
\item $\breve d_q^*|_{H_{k-1,q+1}}= k S_k\otimes \id_{H^{\wedge q}}, \quad \breve d_q^* (H_{k-1,q+1}) \subset H_{k, q}$.
\end{enumerate}
Using (b) and (c), we can check that $\breve d_{q}$ and $\breve d_{q}^*$ are adjoint to each other. 
\begin{lem}
\label{lem:diagram} 
The diagram (\ref{figure:psi_q}) commutes, and  
\begin{equation}
\label{eq:breve_d_q}
\breve d_q [(\exp\odot h)\otimes x]= (\exp\odot h)\otimes (h\wedge x), \quad \forall h\in H, x\in H^{\wedge q}.
\end{equation}
\end{lem}
\begin{proof}
Given any $h\in H$ and $x\in H^{\wedge q}$, we verify that 
\begin{eqnarray*}
\breve d_q[(\exp\odot h)\otimes x]
&=& \breve d_q [\sum_{k=0}^{\infty}\frac{1}{k!}h^{\odot k} \otimes x]\\
&=& \sum_{k=1}^{\infty}\frac{1}{(k-1)!}h^{\odot (k-1)} \otimes h\wedge x\\
&=& (\exp\odot h)\otimes (h\wedge x).
\end{eqnarray*}
The definitions (\ref{eq:psi}) and (\ref{eq:psi_q}) give
\[
\Psi_q [(\exp\odot h)\otimes x] = \exp(I(h)-\frac{1}{2}\|h\|^2)\otimes x, \quad\forall h\in H, x\in H^{\wedge q}, 
\]
so we have
\begin{eqnarray*}
d_q^{\sharp}\Psi_q [(\exp\odot h)\otimes x]
&=& d_q^{\sharp}[\exp(I(h)-\frac{1}{2}\|h\|^2)\otimes x]\\
&=& (q+1)A_{q+1}[\exp(I(h)-\frac{1}{2}\|h\|^2)\otimes h\otimes x]\\
&=& \exp(I(h)-\frac{1}{2}\|h\|^2)\otimes (h\wedge x)\\
&=& \Psi_{q+1} [(\exp\odot h)\otimes (h\wedge x)]\\
&=& \Psi_{q+1}\breve d_q [(\exp\odot h)\otimes x].
\end{eqnarray*}
Therefore, the diagram indeed commutes.
\end{proof}
To understand how the operators $\breve d_q$ and  $\breve d_q^{*}$ interact when they map into different components of $H_{k, q}$ for varying integers $k$ and $q$, we quote the following result of J. Rawnsley, a version of which appeared in \cite{bieliavsky2005symplectic} for the case of $H$ being a finite-dimensional vector space. 
The proof given below was shown to us by J. Rawnsley; it is simple and instructive, and does not depend on the dimension of $H$. 
Let's fix  $k$ and $q$, and set $n=k+q$. 
\begin{prop}[Rawnsley]
\label{prop:weitzenbock}
The following sequences are exact:
\[
0 \rightarrow H^{\odot n}
\xrightarrow{\breve d_0} \cdots \xrightarrow{\breve d_{q-1}} 
H_{k, q}
\xrightarrow{\breve d_q} 
H_{k-1, q+1}
\xrightarrow{\breve d_{q+1}} \cdots \xrightarrow{\breve d_n} 
H^{\wedge n} \rightarrow 0
\]
and 
\[
0 \rightarrow H^{\wedge n}
\xrightarrow{\breve d_n^*} \cdots \xrightarrow{\breve d_{q+1}^*} 
H_{k-1, q+1}
\xrightarrow{\breve d_q^*} 
H_{k, q}
\xrightarrow{\breve d_{q-1}^*} \cdots \xrightarrow{\breve d_0^*} 
H^{\odot n} \rightarrow 0.
\]
\end{prop}
\begin{proof}
First recall from Property (a) that $\breve d_{q+1}\circ \breve d_{q} =0$ and $\breve d_{q}^*\circ \breve d_{q+1}^* =0$.  
Given any $h_1, \cdots, h_k, x_1, \cdots, x_q \in H$, we have
\begin{eqnarray*}
& & \breve d_{q}^*\circ \breve d_{q} (h_1\odot \cdots \odot h_k \otimes  x_1\wedge \cdots \wedge x_q )\\
&=& \breve d_{q}^*[ \sum_{j=1}^k h_1\odot \cdots \odot \hat h_j \odot \cdots \odot h_k\otimes  h_j \wedge x_1\wedge\cdots\wedge x_q]\\
&=& \sum_{j=1}^k [ h_1\odot \cdots \odot \hat h_j \odot \cdots \odot h_k\odot  h_j \otimes  x_1\wedge\cdots\wedge x_q\\
& &+ \sum_{i=1}^{q} (-1)^{i} h_1\odot \cdots \odot \hat h_j \odot \cdots \odot h_k\odot x_i\otimes  h_j \wedge x_1\wedge\cdots\wedge\hat x_i\wedge\cdots\wedge x_q], 
\end{eqnarray*}
and 
\begin{eqnarray*}
& &\breve d_{q-1}\circ \breve d_{q-1}^*(h_1\odot \cdots \odot h_k \otimes  x_1\wedge \cdots \wedge x_q )\\
&=& \breve d_{q-1}[\sum_{i=1}^q (-1)^{i-1} h_1\odot \cdots \odot h_k\odot x_i \otimes x_1\wedge\cdots\wedge\hat x_i\wedge\cdots\wedge x_q]\\
&=&\sum_{i=1}^q (-1)^{i-1} [ h_1\odot \cdots \odot h_k\otimes x_i \wedge x_1\wedge\cdots\wedge\hat x_i\wedge\cdots\wedge x_q\\
& &+ \sum_{j=1}^k h_1\odot \cdots \odot \hat h_j \odot \cdots \odot h_k\odot x_i \otimes h_j\wedge x_1\wedge\cdots\wedge\hat x_i\wedge\cdots\wedge x_q].
\end{eqnarray*}
Therefore, we have the following identity on $H_{k, q}$: 
\begin{equation}
\label{eq:breve_weitzenbock}
\breve d_{q}^*\circ \breve d_{q} + \breve d_{q-1}\circ \breve d_{q-1}^* = (k + q) \id.
\end{equation}
So $\breve d$ and $\breve d^*$ are invertible on the kernel of each other, proving exactness.
\end{proof}
\begin{re}
In particular, Proposition \ref{prop:weitzenbock} shows that $\breve d$ has closed range. 
\end{re}
\begin{re}
Since $\Psi_q$'s are isomorphisms for all $q\in\mathbb{N}$, Proposition \ref{prop:weitzenbock} shows that the kernel of $d$ and the kernel of $d^*$ are disjoint, which leads to the vanishing result in Theorem \ref{th:shigekawa}.   
\end{re}
\begin{re}
It is not difficult to notice the similarity between 
equation (\ref{eq:breve_weitzenbock}) and Shigekawa's Weitzenb\"ock identity (\ref{eq:weitzenbock}). Recall that 
$L=D^*D$ corresponds to the number operators on the Fock space and, when acting on elements of $H_{k, q}$, gives the term $k\id$ on the right-hand side of (\ref{eq:breve_weitzenbock}), 
which explains the difference between (\ref{eq:breve_weitzenbock}) and (\ref{eq:weitzenbock}). 
\end{re}
A bit more can be said of the kernels and images of $\breve d$ and $\breve d^*$ in the Fock space setting; note that the  images of $\breve d$ and $\breve d^*$ correspond to the exact and co-exact forms, respectively, in the abstract Wiener space. 
For a fixed $k$, the exact sequences in Proposition \ref{prop:weitzenbock} can be viewed as 
the restriction of the two sequences in Theorem \ref{th:shigekawa}. 
The following lemma gives a direct-sum decomposition of the tensor product $H^{\odot k}\otimes H^{\wedge q}$, considered as a subspace of $H^{\otimes n}$, which shows how the exact sequences split.  But we first need to study a more general subspace of $H^{\otimes n}$. 
Recall that $H_{k,q}=H^{\odot k}\otimes H^{\wedge q}$ is the subspace of elements of $H^{\otimes n}$ symmetric in the \emph{first} $k$ components and skew-symmetric in the last $q$ components. 
Denote by $H^{\odot [k], \wedge [q]}$ the 
closed linear span of elements of $H^{\otimes n}$ symmetric in \emph{any} $k$ components and skew-symmetric in the remaining $q$ components; note that we are not fixing the order of the symmetric and skew-symmetric parts with respect to each other inside the tensor product. 
We use the following shorthand notation from now on: 
$H_{k,q}^+= H_{k,q}\cap H^{\odot [k+1],\wedge [q-1]}$, and 
$H_{k,q}^-= H_{k, q}\cap H^{\odot [k-1], \wedge [q+1]}$.
\begin{lem}
\label{lem:L-R_decom_Hkq}
Given $k, q\in\mathbb{N}$, let $n=k+q$. We have the decomposition \begin{equation}
\label{eq:L-R_decom_Hkq}
H_{k,q}= H_{k,q}^+\oplus H_{k,q}^-,
\end{equation}
where the equality is understood to take place inside 
$H^{\otimes n}$. 
\end{lem}
Lemma \ref{lem:L-R_decom_Hkq} is a special case of Corollary \ref{cor:L-R_decom_H_a} in the next section, for which a representation-theoretic proof is given. 
The decomposition (\ref{eq:L-R_decom_Hkq}) gives a concrete description of the kernels and images of $\breve d$ and $\breve d^*$, and, together with Proposition \ref{prop:weitzenbock}, allows us to state 
an analogue of Theorem \ref{th:shigekawa}:
\begin{enumerate}
\item $H_{k,q}^+ =\image(\breve d_{q-1}|_{H_{k+1,q-1}})= \Ker(\breve d_{q}|_{H_{k,q}})$;
\item $H_{k,q}^-= \image(\breve d_{q}^*|_{H_{k-1,q+1}})= \Ker(\breve d_{q-1}^*|_{H_{k,q}})$; and 
\item $\Ker(\breve d_{q-1}^*|_{ H_{k,q}}) \cap \Ker(\breve d_{q}|_{ H_{k,q}})= \{0\}$.
\end{enumerate}
It is not difficult to verify that $H_{k,q}^+ \subset \Ker(\breve d_{q}|_{H_{k,q}})$, since $\breve d_{q}$ is, as shown by Property (b), basically a skew-symmetrisation operator and is annihilated by any element of $H_{k,q}$ symmetric in at least two of its last $(q+1)$-components.  Similarly we see $H_{k,q}^-\subset \Ker(\breve d_{q-1}^*|_{H_{k,q}})$.  Proposition \ref{prop:weitzenbock} proves the disjointness of the kernels of $\breve d_{q}|_{H_{k,q}}$ and $\breve d_{q-1}^*|_{H_{k,q}}$, so indeed $H_{k,q}^+ =\Ker(\breve d_{q}|_{H_{k,q}})$, $H_{k,q}^-=\Ker(\breve d_{q-1}^*|_{H_{k,q}})$, and the rest is clear. 
The above can be more succinctly described by the following %commutative 
diagram of long and short exact sequences: 
\begin{displaymath}
\xymatrix{
        &                            &0\ar[d]                 &  0                 &  0\ar[d]                           &
\\
        &                            &H_{k+1,q-1}^-\ar[d]\ar@/^/[r]^{\breve d_{q-1}} &H_{k,q}^+\ar[u] \ar@/^/[l]^{\breve d_{q-1}^*}   &H_{k-1,q+1}^-\ar[d]    &
\\
0\ar[r]& H_{n,0}\ar[r]              & \cdots H_{k+1,q-1}\ar@/^/[r]^{\breve d_{q-1}}\ar[d]  & H_{k,q} \ar@/^/[l]^{\breve d_{q-1}^*}  \ar@/^/[r]^{\breve d_{q}} \ar[u] & H_{k-1,q+1}\ar[r]\ar[d]\ar@/^/[l]^{\breve d_{q}^*}     \cdots&  H_{0,n}\ar[r] & 0.
\\
        &                            &H_{k+1,q-1}^+\ar[d]     & H_{k,q}^-\ar[u]\ar@/^/[r]^{\breve d_{q}} &H_{k-1,q+1}^+\ar[d] \ar@/^/[l]^{\breve d_{q}^*}      & 
\\
        &                            &0                      &  0\ar[u]              &  0                      & 
} 
\end{displaymath}
%%%%%%%%%%%%%%
\section{A Representation-Theoretic Proof}
\label{sec:decomp}
To prove Lemma \ref{lem:L-R_decom_Hkq}, we first note that the symmetric group of degree $n$, $\mathfrak{S}_n$, acts naturally on $H^{\otimes n}$ by permuting the $n$ components. 
The vector subspace $H_{k,q}$ is not an $\mathfrak{S}_n$-invariant subspaces of $H^{\otimes n}$, but $H^{\odot [k], \wedge [q]}$ is. 
\par
To give a more specific description of $H^{\odot [k], \wedge [q]}$, let $\mathbf{n}=\{1, 2, \cdots, n\}$, and 
define the set of the $k$-subsets of $\mathbf{n}$
\[
\mathbf{n}^{[k]} = \{\mathbf{a}\subseteq \mathbf{n}| |\mathbf{a}|=k\}.
\]
For each $\mathbf{a}\in \mathbf{n}^{[k]}$, denote by $\mathbf{a}^c$ its complement in $\mathbf{n}$, and by $H^{\odot \mathbf{a}, \wedge \mathbf{a}^c}$ the subspace of elements of $H^{\otimes n}$ symmetric in the chosen components, specified by $\mathbf{a}$, and skew-symmetric in the remaining components, specified by $\mathbf{a}^c$.  Similarly we have $H^{\odot \mathbf{a}, \otimes \mathbf{a}^c}$, the subspace of elements of $H^{\otimes n}$ only restricted to be symmetric in the chosen components specified by $\mathbf{a}$, and $H^{\wedge \mathbf{a}, \otimes \mathbf{a}^c}$, the subspace of elements of $H^{\otimes n}$ skew-symmetric in the chosen components specified by $\mathbf{a}$. For example, $H_{k,q}=H^{\odot k}\otimes H^{\wedge q}$ correspond to the choice of $\mathbf{a}=\{1, \cdots, k\}$, and hence $\mathbf{a}^c=\{k+1, \cdots, n\}$, so in this notation it can be written as $H^{\odot \{1, \cdots, k\}, \wedge \{k+1, \cdots, n\}}$.
\par
The subspace $H^{\odot [k], \wedge [q]}$ is the span  
of all $C(n, k)$ 
subspaces $H^{\odot \mathbf{a}, \wedge \mathbf{a}^c}$, corresponding to the $C(n, k)$ possible choices of $\mathbf{a}$ in $\mathbf{n}^{[k]}$, of  
elements invariant under a permutation of a specific set of $k$ variables, and anti-invariant under a permutation of the rest. 
For each choice, say, $\mathbf{a}$ and hence $\mathbf{a}^c$, of the $k$ and $q$ variables, the action of the corresponding $\mathfrak{S}_k\times \mathfrak{S}_q$ stabilises $H^{\odot \mathbf{a}, \wedge \mathbf{a}^c}$, 
while the elements of $\mathfrak{S}_n/(\mathfrak{S}_k\times \mathfrak{S}_q)$ permute the spaces $H^{\odot \mathbf{a}, \wedge \mathbf{a}^c}$ with different choices of $\mathbf{a}$'s. 
The structure of $H^{\odot [k], \wedge [q]}$ is similar to that of the representation induced from $H^{\odot \mathbf{a}, \wedge \mathbf{a}^c }$, but the spaces $H^{\odot \mathbf{a}, \wedge \mathbf{a}^c }$ with different choices of $\mathbf{a}$'s in general can intersect non-trivially, and therefore may not form a direct sum. \par
Recall the symmetrisation and skew-symmetrisation operators $S_{n}$ and $A_{n}$, which project elements of $H^{\otimes n}$ onto the closed subspaces  $H^{\odot n}$ and $H^{\wedge n}$, respectively. 
Corresponding to an element $\mathbf{a}\in \mathbf{n}^{[k]}$, we define the operator $S_{\mathbf{a}}:H^{\otimes n}\rightarrow H^{\odot \mathbf{a}, \otimes \mathbf{a}^c}$, which symmetrises any $n$-tensor in its $k$ components specified by $\mathbf{a}$, 
and the operator $A_{\mathbf{a}}:H^{\otimes n}\rightarrow H^{\wedge \mathbf{a}, \otimes \mathbf{a}^c}$, which skew-symmetrises any $n$-tensor in its $k$ components specified by $\mathbf{a}$. To be more precise, for 
$h\in H^{\otimes n}$ and $\rho\in \mathfrak{S}_k$, denote by 
$\rho^{\mathbf{a}}h$ the element of $H^{\otimes n}$ that has its $\mathbf{a}$ components permuted by $\rho$ and the remaining components fixed, so $S_{\mathbf{a}}$ and  $A_{\mathbf{a}}$ are defined, respectively,  by
\[
S_{\mathbf{a}}h 
=\frac{1}{k!} \sum_{\rho \in \mathfrak{S}_k} \rho^{\mathbf{a}}h,
\quad \mbox{ and } \quad
A_{\mathbf{a}}h 
=\frac{1}{k!} \sum_{\rho \in \mathfrak{S}_k} \sgn(\rho) \rho^{\mathbf{a}}h. 
\]
From our earlier discussion, it is clear that $H^{\odot \mathbf{a}, \otimes \mathbf{a}^c}$ is the image of $H^{\otimes n}$ under $S_{\mathbf{a}}$, $H^{\wedge \mathbf{a}, \otimes \mathbf{a}^c}$ is the image under $A_{\mathbf{a}}$, and $H^{\odot \mathbf{a}, \wedge \mathbf{a}^c}$ is the image under $S_{\mathbf{a}}A_{\mathbf{a}^c}$. 
\par
If $\{e_i\}_{i\in\mathbb{N}}$ is a complete orthonormal basis of $H$, we have correspondingly $\{e_{i_1}\otimes\cdots\otimes e_{i_n}\}_{i_1, \cdots, i_n = 1}^\infty$ as a complete orthonormal basis of $H^{\otimes n}$.  
We can choose a basis of $H^{\odot n}$ whose elements are of the form $e_{i_1}\wedge\cdots\wedge e_{i_n}$, with $i_1, \cdots, i_n$ all integers. 
Similarly, we also have a basis of $H^{\wedge n}$ with elements of the form $e_{j_1}\wedge\cdots\wedge e_{j_n}$, 
where $j_1, \cdots, j_q$ are \emph{distinct} integers. For $H^{\odot k}\otimes H^{\wedge q}$, we can similarly take basis elements of the form 
\begin{equation}
\label{eq:Hkq_basis}
e_{i_1}\wedge\cdots\wedge e_{i_k}\otimes e_{j_1}\wedge\cdots\wedge e_{j_q}, 
\end{equation}
where each of the indices $i_1, \cdots, i_k$, $j_1, \cdots$ and $j_{q}$ run from $1$ to $\infty$, and the $j$'s have to be all distinct. 
\par
For $H^{\odot [k], \wedge [q]}$, its basis elements can be chosen almost the same way as in (\ref{eq:Hkq_basis}), but the positions of the components which are symmetric and those which are skew-symmetric depend on one of the $C(n, k)$ choices from $\mathbf{n}^{[k]}$. 
A given basis element of $H^{\odot k}\otimes H^{\wedge q}$ of the form (\ref{eq:Hkq_basis}), say, $b$, corresponds to two specific collections of basis elements of $H$, counted with multiplicity:
\[
E_{b_i} = 
\{e_{i_1}, \cdots, e_{i_k}\}, \quad\mbox{ and } 
E_{b_j} =
\{ e_{j_1},\cdots,e_{j_q}\}.
\]
The vector $b$ also corresponds to the element $\mathbf{a}_b=\{1, \cdots, k\}$ in $\mathbf{n}^{[k]}$.  
Similarly, a basis element of $H^{\odot [k], \wedge [q]}$ involves firstly the choice of two  
collections of basis elements of $H$, counted with multiplicity, where one set (of $k$ elements specified by the $i$-indices) forms the symmetric part of the basis element, and the other set (of $q$ distinct elements specified by the $j$-indices) forms the skew-symmetric part; and secondly the choice of an element in $\mathbf{n}^{[k]}$, for the positioning of the $k$ symmetric components. 
For $b\in H^{\odot k}\otimes H^{\wedge q}$ as in (\ref{eq:Hkq_basis}), 
since the action of $\mathfrak{S}_n$ permutes the $n$ components of $b$, the orbit $O_b$ of $b$ under the action of $\mathfrak{S}_n$ covers all the $C(n, k)$ possibilities of the positioning. 
We can therefore enumerate all our basis elements of $H^{\odot [k], \wedge [q]}$ by going through the basis elements of $H^{\odot k}\otimes H^{\wedge q}$ of the form (\ref{eq:Hkq_basis}) (for our purpose, we don't need to worry about the possible repetitions). 
Denote by $V_{O_b}$ the span of the vectors in $O_b$, which is a subspace of $H^{\odot [k], \wedge [q]}$.  We have thus proved the following 
\begin{lem}
\label{lem:Hkq_span}
Given any $k, q\in\mathbb{N}$,  
we have 
\begin{equation}
\label{eq:H[kq]}
H^{\odot [k], \wedge [q]} = \textup{Span}\left(\bigcup_b V_{O_b}\right),
\end{equation}
where the union is taken over all basis elements 
of a complete orthonormal basis of $H^{\odot k}\otimes H^{\wedge q}$.
\end{lem}
In the sequel, we will often study the basis elements of $H^{\odot [k], \wedge [q]}$ through those of $H^{\odot k}\otimes H^{\wedge q}$, which give easier notation for explicit expressions. 
\par
For a fixed basis element $b$ of $H^{\odot [k], \wedge [q]}$, 
there is a subgroup of $\mathfrak{S}_n$ isomorphic to $\mathfrak{S}_k\times \mathfrak{S}_q$ whose representation on the one-dimensional space spanned by $b$ is $[k]\sharp[1^q]$, the outer tensor product of $[k]$ and $[1^q]$, an irreducible representation of $\mathfrak{S}_k\times \mathfrak{S}_q$ (e.g., see Section 2.3 of \cite{james1981representation}). 
The irreducible representation $[k]\sharp[1^q]$ on Span($b$) induces into $\mathfrak{S}_n$ the representation $[k][1^q]$ on $V_{O_b}$.
A simple application of the Littlewood-Richardson rule (Theorem 2.8.13, or more directly, 
Corollary 2.8.14, of \cite{james1981representation}) yields the following decomposition into irreducible constituents:
\begin{equation}
\label{eq:L-R_decom}
[k][1^q] = [k+1, 1^{q-1}]\oplus[k, 1^q].
\end{equation}
Hence, every subspace $V_{O_b}$  
splits into a direct sum of two irreducible components 
\begin{equation}
\label{eq:Serre_decom}
V_{O_b} = V^+_{O_b} \oplus V^-_{O_b}, 
\end{equation}
where $V^+_{O_b}$ and $V^-_{O_b}$ correspond to $[k+1, 1^{q-1}]$ and $[k, 1^q]$, respectively. 
\par
Suppose we have another basis element $b'$ of $H^{\odot [k], \wedge [q]}$, 
with a corresponding orbit $O_{b'}$ and an associated subspace $V_{O_{b'}}$. 
As in the discussion earlier, in terms of the basis elements of $H$ appearing in the expression of $b$ and $b'$, we have 
two sets
$E_{b_i} = \{e_{i_1}, \cdots, e_{i_k}\}$ and 
$E_{b_j} = \{ e_{j_1},\cdots,e_{j_q}\}$, 
where the $i$'s and $j$'s are integers and the $j$'s are all distinct; and similarly 
$E_{b'_i} = \{e_{i'_1}, \cdots, e_{i'_k}\}$ and 
$E_{b'_j} = \{ e_{j'_1},\cdots,e_{j'_q}\}$,
where the $i'$'s and $j'$'s are integers and the $j'$'s are all distinct. 
\par
Observe that the orbits $O_b$ and $O_{b'}$ are disjoint and the spaces $V_{O_b}$ and $V_{O_{b'}}$ have a trivial intersection, 
as long as the sequences $(E_{b_i}, E_{b_j})$ and $(E_{b'_i}, E_{b'_j})$ differ. 
Therefore, each $V_{O_b}$ intersects at most finitely many other subspaces $V_{O_{b'}}$.  If we have a non-trivial element $v\in V_{O_b}\cap V_{O_{b'}}$, the orbit of $v$ under the action of $\mathfrak{S}_n$ spans an invariant subspace of both $V_{O_b}$ and $V_{O_{b'}}$. Our earlier discussion shows that, either these two spaces coincide, i.e., 
\[
V_{O_b}=V_{O_{b'}}(=V_{O_b}\cap V_{O_{b'}}),
\]
or their intersection corresponds to one of the two components in (\ref{eq:L-R_decom}), i.e., $[k+1, 1^{q-1}]$ and $[k, 1^q]$,
so in terms of (\ref{eq:Serre_decom}),
\[
V_{O_b}\cap V_{O_{b'}} = V^+_{O_b} \mbox{ or }  V^-_{O_b}.
\]
In summary, the action of $\mathfrak{S}_n$ splits the collection of our basis elements of $H^{\odot [k], \wedge [q]}$ 
into disjoint subsets, each of which spans a vector subspace of $H^{\odot [k], \wedge [q]}$, 
which is a copy of the representation $[k][1^q]$ of $\mathfrak{S}_n$. Any non-trivial intersection of these vector subspaces, when they do not coincide, is limited to be one of the two irreducible components, as shown above.  Therefore, the space $H^{\odot [k], \wedge [q]}$ 
is 
made of infinitely many finite-dimensional isomorphic representations of $\mathfrak{S}_n$, each isomorphic to $[k][1^q]$, mostly disjoint from the rest, but possibly intersecting a few along its irreducible components.  
\par
This discussion enables us to state the following decomposition of $H^{\odot [k], \wedge [q]}$ 
inside $H^{\otimes n}$. 
\begin{lem}
\label{lem:L-R_decom_H}
Given any $k, q\in\mathbb{N}$, let $n=k+q$. Then we have 
\begin{equation}
\label{eq:L-R_decom_H}
H^{\odot [k], \wedge [q]}= (H^{\odot [k], \wedge [q]}\cap H^{\odot [k+1], \wedge [q-1]})
\oplus (H^{\odot [k], \wedge [q]} \cap H^{\odot [k-1], \wedge [q+1]}),
\end{equation}
where the equality is understood to take place inside $H^{\otimes n}$. 
\end{lem}
\begin{proof}
As mentioned above, $\mathfrak{S}_n$ acts naturally on $H^{\otimes n}$ by permuting the $n$ components. 
The vector subspaces in question, i.e., $H^{\odot [k], \wedge [q]}$, $H^{\odot [k+1], \wedge [q-1]}$, and $H^{\odot [k-1], \wedge [q+1]}$, as well as their intersections appearing in (\ref{eq:L-R_decom_H}), are $\mathfrak{S}_n$-invariant subspaces of $H^{\otimes n}$. 
\par 
Lemma \ref{lem:Hkq_span} and the discussion afterwards show that $H^{\odot [k], \wedge [q]}$ consists of subspaces isomorphic to the representation $[k][1^q]$ of $\mathfrak{S}_n$, each intersecting finitely many others, with the non-trivial intersection being one of the two irreducible components of $[k][1^q]$. 
The same statements can be made for $H^{\odot [k+1], \wedge [q-1]}$ and $H^{\odot [k-1], \wedge [q+1]}$, but replacing $[k][1^q]$ with %the representations 
$[k+1][1^{q-1}]$ and $[k-1][1^{q+1}]$, respectively. 
\par
Similar to (\ref{eq:L-R_decom}), we also have the Littlewood-Richardson decompositions for $[k+1][1^{q-1}]$ and $[k-1][1^{q+1}]$, i.e., 
\[
[k+1][1^{q-1}] =[k+2, 1^{q-2}] \oplus [k+1, 1^{q-1}]
\]
and 
\[
[k-1][1^{q+1}] = [k, 1^{q}]\oplus[k-1, 1^{q+1}],
\]
respectively. Observe that $[k][1^q]$ has exactly one irreducible component in common with $[k+1][1^{q-1}]$, which is $[k+1, 1^{q-1}]$, and exactly one with $[k-1][1^{q+1}] $, which is $[k-1, 1^{q+1}]$, and no other ones. 
\par
Now Lemma \ref{lem:Hkq_span} implies that $H^{\odot [k+1], \wedge [q-1]}$ and $H^{\odot [k-1], \wedge [q+1]}$ have only a trivial intersection; indeed, the intersection would have to be $\mathfrak{S}_n$-invariant, but (\ref{eq:H[kq]}) and (\ref{eq:L-R_decom}) show that it has to be trivial.  Therefore, we only need to show that $H^{\odot [k], \wedge [q]}$ does intersect $H^{\odot [k+1], \wedge [q-1]}$ and $H^{\odot [k-1], \wedge [q+1]}$ separately, in a manner corresponding to the way $[k][1^q] $ intersects with $[k+1][1^{q-1}]$ and $[k-1][1^{q+1}]$, which then gives us the direct sum as in (\ref{eq:L-R_decom_H}). 
\par
Again we can look at an arbitrary basis element $b$ of the form (\ref{eq:Hkq_basis}), 
and its associated vector subspace $V_{O_b}\subset H^{\odot [k], \wedge [q]}$. All we need is to find two vectors, 
\[
v^+\in V_{O_b}\cap H^{\odot [k+1], \wedge [q-1]}, \quad\mbox{ and }
v^-\in V_{O_b}\cap H^{\odot [k-1], \wedge [q+1]},
\]
since the two disjoint invariant subspaces of $V_{O_b}$, $V_{O_b}\cap H^{\odot [k+1], \wedge [q-1]}$ and $V_{O_b}\cap H^{\odot [k-1], \wedge [q+1]}$, have to correspond to the $[k+1, 1^{q-1}]$ and $[k, 1^{q}]$ components, respectively.  
\par
Denote by $\tau_{i,j}$ the transposition operator on $H^{\otimes n}$, which acts by exchanging the $i$-th and $j$-th components of a tensor product: i.e., given any $h_1, \cdots, h_n \in H$, 
\begin{equation*}
\tau_{i,j}(h_1\otimes\cdots\otimes h_i\otimes\cdots\otimes h_j\otimes\cdots\otimes h_n)
=h_1\otimes\cdots\otimes h_j\otimes\cdots\otimes h_i\otimes\cdots\otimes h_n. 
\end{equation*}
Now we can express the result of swapping the $l$-th and $(k+1)$-th components of $b$ as $\tau_{l, k + 1} b$, 
which is an element of $O_b$ and  
of $H^{\odot \mathbf{a}(l), \wedge \mathbf{a}(l)^c}$, where we define $\mathbf{a}(l) =\{1, \cdots, \hat l, \cdots, k, k+1\}\in \mathbf{n}^{[k]}$, and $l$ ranges from $1$ to $k+1$. 
Similarly,  for each $m=k+1, \cdots, n$, we have $\tau_{k, m} b$, 
an element of $O_b$ and 
of $H^{\odot \mathbf{a}<m>, \wedge \mathbf{a}<m>^c}$, with $\mathbf{a}\!<\!m\!> =\{1, \cdots, k-1, m\}\in \mathbf{n}^{[k]}$. 
We conclude the proof by setting 
\[
v^+ =\frac{1}{k+1} \sum_{l=1}^{k+1}\tau_{l, k + 1} b
\]
and 
\[
v^-= \frac{1}{q+1} (b - \sum_{m=k+1}^{n}\tau_{k,m} b).\qedhere
\]
\end{proof}
\begin{cor}
\label{cor:L-R_decom_H_a}
Given any $\mathbf{a}\in \mathbf{n}^{[k]}$, we have 
\begin{equation}
\label{eq:L-R_decom_H_a}
H^{\odot \mathbf{a}, \wedge \mathbf{a}^c}= (H^{\odot \mathbf{a}, \wedge \mathbf{a}^c}\cap H^{\odot [k+1], \wedge [q-1]})
\oplus (H^{\odot \mathbf{a}, \wedge \mathbf{a}^c}\cap H^{\odot [k-1], \wedge [q+1]}).
\end{equation}
\end{cor}
\begin{proof}For any $g\in H^{\odot \mathbf{a}, \wedge \mathbf{a}^c}\subset H^{\odot [k], \wedge [q]}$, we have 
\[
g= A_{\mathbf{a}^c}S_{\mathbf{a}} g.
\]
Lemma \ref{lem:L-R_decom_H} gives a direct-sum decomposition for $g$
\[
g=\tilde g + \hat g,
\]
with 
$\tilde g\in H^{\odot [k], \wedge [q]}\cap H^{\odot [k+1], \wedge [q-1]}$, and $\hat g\in H^{\odot [k], \wedge [q]}\cap H^{\odot [k-1], \wedge [q+1]}$. So we have
\[
g=A_{\mathbf{a}^c}S_{\mathbf{a}}g=A_{\mathbf{a}^c}S_{\mathbf{a}}\tilde g + A_{\mathbf{a}^c} S_{\mathbf{a}}\hat g, 
\]
where 
\[
A_{\mathbf{a}^c}S_{\mathbf{a}}\tilde g \in (H^{\odot [k], \wedge [q]}\cap H^{\odot [k+1], \wedge [q-1]}) \cap H^{\odot \mathbf{a}, \wedge \mathbf{a}^c}
= H^{\odot \mathbf{a}, \wedge \mathbf{a}^c}\cap H^{\odot [k+1], \wedge [q-1]},
\]
and 
\[
A_{\mathbf{a}^c}S_{\mathbf{a}}\hat g\in (H^{\odot [k], \wedge [q]}\cap H^{\odot [k-1], \wedge [q+1]} )
\cap H^{\odot \mathbf{a},  \wedge \mathbf{a}^c}
= H^{\odot \mathbf{a}, \wedge \mathbf{a}^c}\cap H^{\odot [k-1], \wedge [q+1]}.
\] 
By the uniqueness of the direct-sum decomposition, we have $\tilde g=A_{\mathbf{a}^c}S_{\mathbf{a}}\tilde g$ and $\hat g=A_{\mathbf{a}^c}S_{\mathbf{a}}\hat g$, and the proof is complete. 
\end{proof}
W. Hamernik \cite{hamernik1976specht} gave a proof of the decomposition (\ref{eq:L-R_decom_H}) for the case where $H$ is a finite-dimensional vector space. 
%%%%%%%%%%%%%%
\bibliographystyle{amsplain}
\bibliography{../../wt}        
\end{document}